\documentclass[11pt]{article}
\usepackage{amsmath, amsthm, amssymb}
\usepackage[english]{babel}
\usepackage{graphics}
\usepackage{amsfonts}
\usepackage{amssymb}
\usepackage{mathrsfs}
\usepackage{enumerate}
\usepackage{algorithm}
\usepackage{algorithmic}
\usepackage{multirow}
\usepackage{booktabs}
\usepackage{float}
\usepackage{multicol}
\usepackage[margin=0.5in]{geometry} 
\usepackage{hyperref}
\usepackage{graphicx}
\usepackage{color}
\usepackage{pgfplots}
\usepackage{tikz}
\usepackage[T1]{fontenc}
\usepackage{anyfontsize}
\usetikzlibrary{patterns}
\usetikzlibrary{calc}
\tikzset{
	convexset/.style = {line width = 0.5 pt, fill=lightgray, opacity=0.3},
	ext/.style = {circle, inner sep=0pt, minimum size=2pt, fill=black},
	segment/.style = {line width = 0.75 pt}
}

\theoremstyle{definition}

\newtheorem{proposition}{Proposition}
\newtheorem{example}{Example}
\newtheorem{theorem}{Theorem}
\newtheorem{remark}{Remark}

\def\int{\textrm{int}}

\def\R{\mathbb{R}}

\def\N{\hbox{{\rm I}\kern-.2em\hbox{{\rm N}}}}

\date{}

\begin{document}

\title{Douglas-Rachford splitting algorithm for projected solution of quasi  variational inequality with non-self constraint map}
\author{
M. Ramazannejad \thanks {Universit{\`a} Cattolica del Sacro Cuore, Milano \tt( maede.ramazannejad@unicatt.it)}}
\maketitle

\begin{abstract}
\noindent 
In this paper, we present a Douglas-Rachford splitting algorithm within a Hilbert space framework that yields a projected solution for a quasi-variational inequality. This is achieved under the conditions that the operator associated with the problem is Lipschitz continuous and strongly monotone. The proposed algorithm is based on the interaction between the resolvent operator and the reflected resolvent operator.
\end{abstract}

\section{Introduction}
We present an algorithmic solution for the quasi variational inequality 
\begin{eqnarray}\label{QVIwithNonCons1}
	\text{find}\quad x_0\in P_C(z_0)\quad\text{with}\quad z_0\in \Phi(x_0)\quad\text{such that}\quad\langle T(z_0), y-z_0\rangle \geq 0,\quad \forall y\in \Phi(x_0),
\end{eqnarray}
where $C\subseteq\mathcal{H}$ is a nonempty set, $\mathcal{H}$ is a Hilbert space with inner product $\langle ., .\rangle$ and norm $\|.\|$, $P_C:\mathcal{H}\rightrightarrows C$ is the classical metric projection map, $T: \mathcal{H} \rightarrow \mathcal{H}$ is a single-valued operator, and $\Phi: C \rightrightarrows \mathcal{H}$ is a set-valued map, so that for any element $x\in C$, associates a nonempty, closed, and convex set $\Phi(x)\subseteq \mathcal{H}$.

Based on characterization of problem \eqref{QVIwithNonCons1}, it can take the following forms:
\begin{eqnarray}\label{QVIwithNonCons2}
	\text{find}\quad x_0\in P_C(z_0)\quad\text{such that}\quad 0\in (T+ N_{\Phi(x_0)})(z_0),
\end{eqnarray}
where $N_{\Phi(x_0)}$ is the normal cone. And
\begin{eqnarray}\label{QVIwithNonCons3}
	\text{find}\quad x_0\in P_C(z_0)\quad\text{such that for all ~} \xi\geq 0\quad z_0 = P_{\Phi(x_0)}(z_0 - \xi T(z_0)).
\end{eqnarray}

The point $x_0$ in the above three equivalent forms of quasi variational inequality with a non-self constraint map is called a projected solution, where $z_0$ is a solution of the Stampacchia variational inequality $S(T, \Phi(x_0))$.

In special cases if:
\begin{enumerate}
\item[(i)]  $\Phi: C \rightrightarrows C$ is a set-valued self-map then $z_0:= x_0$, and the problem simplifies to the quasi variational inequality problem 
\begin{eqnarray}\label{QVIwithSelfCons}
\text{find}\quad x_0\in \Phi(x_0)\quad\text{such that}\quad\langle T(x_0), y-x_0\rangle \geq 0,\quad \forall y\in \Phi(x_0),
\end{eqnarray}	
\item[(ii)] $\Phi$ is constant, i.e. $\Phi(x)=C$ for all $x\in C$, the problem reduces to the classical variational inequality problem
\begin{eqnarray}\label{VI}
\text{find}\quad x_0\in C\quad\text{such that}\quad\langle T(x_0), y-x_0\rangle \geq 0,\quad \forall y\in C.
\end{eqnarray}
\end{enumerate}
There is a discernible distinction in the outcomes concerning the existence of solutions for variational and quasi-variational inequalities. To illustrate, in the case where the mapping $T$ exhibits Lipschitz continuity and strong monotonicity within a closed and convex set, the variational inequality admits a single, unique solution (refer to \cite{Ans}). However, when considering QVI \eqref{QVIwithSelfCons}, the assurance of solution existence is provided by \cite[Theorem 9]{NoorOettli}.

This paper has the following structure. Section 2 provides a review of the necessary definitions. In Section 3, we clarify the assumptions that underpin the projected solution of Problem \eqref{QVIwithNonCons1}. Initially, we characterize these assumptions and explore their interconnections. Following this thorough characterization, we introduce the Douglas-Rachford splitting algorithm, specifically tailored for addressing Problem \eqref{QVIwithNonCons1}. The first subsection is devoted to this detailed presentation. Additionally, for ease of comparison, we include the specifics of the algorithm as presented in \cite{BianchiMiglierinaMaede}. This allows for a straightforward comparison between the two algorithms. The second subsection is dedicated to a numerical experiment that compares the speed and efficiency of the two aforementioned algorithms. Further details regarding the computations involved in the numerical experiment are provided in the Appendix.

\section{Preliminaries}
Let us recall a few classical definitions: 
\begin{enumerate}
\item[$\bullet$] Let $f: \mathcal{H} \rightarrow \mathbb{R}\cup \{+\infty\}$ be a proper, convex and lower semicontinuous function. The subdifferential of $f$ at $x_0\in dom(f)$ is the set
$$
\partial f(x_0):=\{v\in \mathcal{H}: f(x)\geq f(x_0) + \langle v, x-x_0\rangle, \forall x\in \mathcal{H}\}.
$$
Clearly, $\partial f:\mathcal{H} \rightrightarrows \mathcal{H}$ is a monotone operator in the sense that for all $x, y \in \mathcal{H}$ and all $v\in \partial f(x)$, $w\in \partial f(y)$ it holds
$$
\langle v-w, x-y\rangle \geq 0.
$$
According to Rockafellar's well-known classical result, $\partial f$ is in fact maximal monotone. This means that the graph of the monotone operator $\partial f$ is not properly contained in the graph of any other monotone operator on $\mathcal{H}$.
\item[$\bullet$] Let $A: \mathcal{H} \rightrightarrows \mathcal{H}$ and $I$ be the identity map on $\mathcal{H}$. 
\begin{enumerate}
\item[$\circ$] The operator $A$ is nonexpansive if one has
\begin{eqnarray*}
\|x^* - y^*\| \leq \|x-y\|,
\end{eqnarray*}
for all $x^*\in A(x),~y^*\in A(y)$.
\item[$\circ$] For $\lambda > 0$, the \textit{resolvent} of operator $A$ is the operator
\begin{eqnarray*}
J_{\lambda A} =(I+\lambda A)^{-1}.
\end{eqnarray*}
The \textit{reflected resolvent} is defined by $ R_\lambda A = 2J_{\lambda A}- I$.  
\end{enumerate}

It is evident that the domain of $J_{\lambda A}$ equals the image of $I+\lambda A$. If $A$ is maximally monotone, Minty's theorem \cite[Theorem 21.1]{BauComb} provides that the resolvent has full domain and as per Corollary 23.10 in \cite{BauComb}, $J_{\lambda A}$ and $ R_{\lambda A}$ are nonexpansive.
\item[$\bullet$] The \textit{proximal operator} $\mathrm{prox}_{\lambda f}(v): \mathcal{H} \rightrightarrows \mathcal{H}$ is defined by
$$
\mathrm{prox}_{\lambda f}(x) = \mathrm{argmin}_{v\in \mathcal{H}}(f(v) + \frac{1}{2\lambda} \|x-v\|^2),
$$
where $f: \mathcal{H} \rightarrow \mathbb{R}\cup \{+\infty\}$ is a proper, convex and lower semicontinuous function. In the special case, when $f$ is the indicator function 
	$$
I_C(x)= \begin{cases}
	0 & x\in C, \\
	+\infty & x\notin C,
\end{cases}
$$
we have
\begin{eqnarray}\label{ProxProj}
\mathrm{prox}_f(x) = \mathrm{argmin}_{v\in C}\|x-v\|^2 = P_C(x).
\end{eqnarray}
\item[$\bullet$] For any $\lambda>0$, the relation between $\mathrm{prox}_{\lambda f}$, $J_{\lambda \partial f}$ and the subdifferential operator $\partial f$ is as follows:
\begin{eqnarray}\label{RelProxRes}
\mathrm{prox}_{\lambda f}= J_{\lambda \partial f},
\end{eqnarray}
The subdifferential of the indicator function at $x$ is the \textit{normal cone}
$$
N_C(x)= \partial I_C(x):= \{p\in\mathcal{H}: \langle p, x\rangle =\sup_{y\in C}\langle p, y\rangle\}=\{p\in\mathcal{H}: \langle p, y-x\rangle\leq 0, \forall y\in C\},
$$
where $C \subset \mathcal{H}$ is a nonempty, closed, and convex set. When this fact is combined with \eqref{ProxProj} and \eqref{RelProxRes}, the result is
\begin{eqnarray}\label{ResNoCoProj}
J_{\lambda N_C} = P_C.
\end{eqnarray}

\item[$\bullet$] The single-valued map $T: \mathcal{H} \rightarrow \mathcal{H}$ is called
\begin{enumerate}
		\item[$\circ$] strongly monotone on $\mathcal{H}$, if there exists $\mu> 0$ such that
		\begin{eqnarray*}
			\langle T(x)-T(y), x - y \rangle \geq \mu \|x\ - y\|^2,\quad\text{for all}\quad x, y\in \mathcal{H}.
		\end{eqnarray*}
		\item[$\circ$]  Lipschitz continuous, if there exists $L> 0$ such that
		\begin{eqnarray*}
			\| T(x)-T(y)\| \leq L \|x-y\|,\quad\text{for all}\quad x, y\in \mathcal{H}.
		\end{eqnarray*}
	\end{enumerate}
\item[$\bullet$] The set-valued map $\Phi: C \rightrightarrows \mathcal{H}$ is Lipschitz continuous on $C\subset\mathcal{H}$ with $L>0$ if
\begin{eqnarray}
\Phi(x)\subset \Phi(y)+ L\|x-y\|\mathbb{B},\quad\forall x, y\in C,
\end{eqnarray}
where $\mathbb{B}$ denotes the closed unit ball in $\mathcal{H}$, or equivalently defined in \cite{HutKed} by
\begin{eqnarray}
d_{H}(\Phi(x), \Phi(y))\leq L \|x-y\|\quad\forall x, y\in C,
\end{eqnarray}
where $d_{H}$ is \textit{Hausdorff metric} defined as
\begin{eqnarray*}
d_{H}(\Phi(x), \Phi(y))&:=& \max\{\sup_{u\in \Phi(x)}d(u, \Phi(y)), \sup_{v\in \Phi(y)}d(v, \Phi(x))\}\\
&=&\sup_{z\in\mathcal{H}} \|d(z, \Phi(x)) -d(z, \Phi(y))\|.
\end{eqnarray*}
\end{enumerate}

\section{Main results}
This section outlines the assumptions that guarantee the projected solution of the problem \eqref{QVIwithNonCons1}. We then present the algorithm that ensures the generated sequence converges to the solution. Finally, we evaluate the convergence of algorithm through an example.

Let the following assumptions hold in the subsequent discussion:
\begin{enumerate}	
\item[(A1)] The set $C\subseteq\mathcal{H}$ is nonempty, closed, and convex.
\item[(A2)] For every $x\in C$, the set $\Phi(x)$ is a nonempty, closed, and convex set.
\item[(A3)] The mapping $T$ is Lipschitz continuous and strongly monotone on $\mathcal{H}$ with constants $L > 0$ and $\mu > 0$, respectively.
\item[(A4)] There exists $l>0$ such that
\begin{eqnarray}\label{contractionproj}
\|P_{\Phi(x)}(z) - P_{\Phi(y)}(z)\|\leq l\|x-y\|,\quad \text{for all}\quad x, y \in C \text{ and } z\in\mathcal{H}. 
\end{eqnarray}
\end{enumerate}
Extensive research has  considered the inequality \eqref{contractionproj}  in the topic of quasi variational inequality problems (see, for instance,\cite{AdLe,JacMij,MijJac,NestScrim,NoorOettli}).

If $\Phi(x)$ remains constant regardless of changes in $x$, we can simplify assumption (A4) by setting $l$ to zero. In this scenario, problem \eqref{QVIwithSelfCons} possesses a unique solution, provided that assumption (A3) is met. This outcome aligns with a well-established result in variational inequalities, as documented in \cite{NoorOettli}.


It is well known that under assumption (A1), the projection mapping $P_C$ is nonexpansive, i.e.
\begin{eqnarray*}
\|P_C(x) - P_C(y)\| \leq \|x-y\|,\quad\text{for all}\quad x, y\in\mathcal{H}.
\end{eqnarray*}
\begin{proposition}\label{RelLipA4}
If the inequality \eqref{contractionproj} is satisfied with constant $L>0$, then the set-valued map $\Phi$ is Lipschitz continuous on $C$ with constant $L$.
\end{proposition}
\begin{proof}
Assume that the inequality \eqref{contractionproj} is satisfied with constant $L>0$. For every  $x, y \in C$ and $z\in \mathcal{H}$, we get
\begin{eqnarray*}
|d(z, \Phi(x)) -d(z, \Phi(y))|&=&|d(z, P_{\Phi(x)}(z)) - d(z, P_{\Phi(y)}(z))| \leq d(P_{\Phi(x)}(z), P_{\Phi(y)}(z))\\
&=&\|P_{\Phi(x)}(z) - P_{\Phi(y)}(z) \| \leq L \|x-y\|,\quad \text{for all}\quad z\in\mathcal{H}. 
\end{eqnarray*}
Then 
\begin{eqnarray*}
d_{H}(\Phi(x), \Phi(y))=\sup_{z\in\mathcal{H}}|d(z, \Phi(x)) -d(z, \Phi(y))| \leq L \|x-y\|,\quad \text{for all}\quad z\in\mathcal{H}. 
\end{eqnarray*}
\end{proof}
You may see that the reverse of the previous proposition is false by looking at the following example.

\begin{example}
Consider  $C:= [0,1]\times \{0\}$ and $\Phi: C\rightrightarrows \mathbb{R}^2$ such that
\begin{equation*}
\Phi(x,0)=\{(u,v)\in \mathbb{R}^2: v= x(1-u); 0\leq u\leq 1 \}
\end{equation*}	
We can easily find that the map $\Phi$ is Lipschitz on $C$ with constant $1$. Indeed,
\begin{eqnarray*}
d_H(\Phi((x,0)), \Phi((y,0))) = \left| x-y\right| ,\quad \forall x,y\in [0,1].
\end{eqnarray*}
Hence 
\begin{eqnarray*}
d_H(\Phi((x,0)), \Phi((y,0))) \leq \|(x,0)- (y,0)\| ,\quad \forall x,y\in [0,1].
\end{eqnarray*}
But if $z=\left(1,2\right)$, we have $P_{\Phi(0,0)}(z)=(1,0)$, $P_{\Phi(1,0)}(z)=(0, 1),$ which imply that 
\begin{eqnarray*}
\|P_{\Phi((0,0))}(z)- P_{\Phi((1,0))}(z)\| = \sqrt{2}>\|(1,0)-(0,0)\|.
\end{eqnarray*}
\end{example}


The relation between inequality \eqref{contractionproj} and Lipschitz continuity of the map $\Phi$ was interestingly characterized by the authors in \cite[Proposition 5.3]{AttWet}. For convenience of the reader, we repeat the proposition here.
\begin{proposition}
Let $C$ and $D$ two nonempty, closed, and convex subsets of a Hilbert space $\mathcal{H}$. Given $x_0\in \mathcal{H}$ and
\begin{eqnarray*}
\rho = \|x_0\| + d(x_0, C) + d(x_0, D),
\end{eqnarray*}
we have the following estimation:
\begin{eqnarray}\label{relIneqLip}
\|P_C(x_0)- P_D(x_0)\|\leq \sqrt{\rho ~ d_{H, \rho}(C, D)},
\end{eqnarray}
where 
$$d_{H, \rho}(C, D):= \max\{\sup_{u\in C\cap \rho B}d(u, D), \sup_{v\in D\cap \rho B}d(C, v)\},$$ 
and $B$ is the unit ball.
\end{proposition}
Based on \eqref{relIneqLip}, if for every $x, y\in \mathcal{H}$, we have $\sqrt{\rho ~ d_{H, \rho}(\Phi(x), \Phi(y))}\leq d_{H}(\Phi(x), \Phi(y))$ and $\Phi$ is Lipschitz continuous with constant $L$, then the inequality \eqref{contractionproj} is satisfied.







In the following proposition, we can determine the assumptions under which the reflected resolvent operator can gain the property of contraction.
\begin{proposition}\cite{Gis}
Let $T: \mathcal{H} \rightarrow \mathcal{H}$ be strongly monotone with constant $\mu$ and L-Lipschitz continuous with $L\geq \mu$. If $\xi>0$, then $R_{\xi T}$ is Lipschitz continuous with modulus 
\begin{eqnarray}
L_{\xi T}:=\sqrt{1-\frac{4\xi\mu}{1+2\xi\mu+\xi^2 L^2}}.
\end{eqnarray}
The value $\xi^*:=\frac{1}{L}$ minimizes $L_{\xi T}$ with respect to $\xi $ and 
\begin{eqnarray*}
L_{\xi^* T}:=\sqrt{\frac{\gamma -1}{\gamma + 1}},
\end{eqnarray*}
where $\gamma:=\frac{L}{\mu}\geq 1$.
\end{proposition}

Inspired by Lemmas 2 and 3 as well as Theorem 2 in \cite{AdLe}, the framework for the projected solution of \eqref{QVIwithNonCons1} as follows.

\begin{proposition}\label{problemsolutionlemma1}
	Let $\xi>0$. If $(z^*, y^*, x^*)$ satisfies 
	\begin{eqnarray}\label{equationsofsolution}
		\begin{cases}
			z^*= J_{\xi N_{\Phi(x^*)}}(y^*)= P_{\Phi(x^*)}(y^*)\\   
			y^* = R_{\xi T}(2z^*-y^*) \\
			x^* = P_C(z^*)
		\end{cases}
	\end{eqnarray}
	then $x^*$ is a projected solution of problem \eqref{QVIwithNonCons1}.
\end{proposition}
\begin{proof}
	We have
	\begin{eqnarray}\label{solutionSTV1}
		y^*= R_{\xi T}(2z^*-y^*)&\Leftrightarrow& y^*= 2 J_{\xi T}(2z^*-y^*)-2z^*+y^*\nonumber\\
		&\Leftrightarrow& y^* = 2(I+\xi T)^{-1}(2z^*-y^*)-2z^*+y^*\\
		&\Leftrightarrow&z^*-y^*= \xi T(z^*)\nonumber
	\end{eqnarray}
	On the other hand,
	\begin{eqnarray}\label{solutionSTV2}
		z^*= J_{\xi N_{\Phi(x^*)}}(y^*)&\Leftrightarrow& z^* = (I+\xi N_{\Phi(x^*)})^{-1}(y^*)\nonumber\\
		&\Leftrightarrow& y^*-z^*\in \xi N_{\Phi(x^*)}(z^*)
	\end{eqnarray}
	The third line of \eqref{equationsofsolution} along with \eqref{solutionSTV1} and \eqref{solutionSTV2} by \eqref{QVIwithNonCons2} show that $x^*$ is a solution of the problem \eqref{QVIwithNonCons1}.
\end{proof}

Let $\rho:\mathcal{H}\times\mathcal{H}\times\mathcal{H}\rightarrow\mathcal{H}\times\mathcal{H}\times\mathcal{H}$ defined as:
$$
\rho(w) = (J_{\xi N_{\Phi(x)}}(y), R_{\xi T}(2J_{\xi N_{\Phi(x)}}(y)-y), P_C(z)),\quad w=(z, y, x)\in \mathcal{H}\times\mathcal{H}\times\mathcal{H},
$$
where $\mathcal{H}\times\mathcal{H}\times\mathcal{H}$ is equipped with the new norm
$$
\|w\|_1 = \alpha\|z\|+ \|y\|+\beta \|x\|
$$
for suitable positive $\alpha$ and $\beta$. 
\begin{proposition}\label{problemsolutionlemma2}
Let $\xi:=\frac{1}{L}$ and $\alpha$, $\beta$ such that
\begin{eqnarray}\label{ParametersAssums}
2l \le \beta \le \alpha L_{\xi T}, \, \,   \alpha + 2 L_{\xi T} < 1
\end{eqnarray}
where $L_{\xi T} = \sqrt {\dfrac{\gamma -1}{\gamma +1}}$ with $1\leq\gamma = \dfrac L {\mu}< \dfrac 5 {3}$.
	Then the function $\rho$ is $\delta-$contractive where
	\begin{eqnarray*}
		\delta:= \alpha + 2 L_{\xi T}.
	\end{eqnarray*}
	Consequently, there exists a unique triplex $(z^*, y^*,x^*)$ such that
	$$
	z^*= J_{\xi N_{\Phi(x^*)}}(y^*),\quad y^* = R_{\xi T}(2z^*-y^*),\quad and\quad x^* = P_C(z^*).
	$$
	In particular, \eqref{QVIwithNonCons1} has a unique projected solution $x^*$.
\end{proposition}

\begin{proof}
	Let $w_1=(z_1, y_1, x_1)$ and $w_2=(z_2, y_2, x_2)$ be given in $\mathcal{H}\times\mathcal{H}\times\mathcal{H}$. Using the contraction property \eqref{contractionproj},  the nonexpansiveness of the resolvent, and \eqref{ResNoCoProj} we obtain
	\begin{eqnarray*}
		\|J_{\xi N_{\Phi(x_1)}}(y_1)-J_{\xi N_{\Phi(x_2)}}(y_2)\|&\leq& \|J_{\xi N_{\Phi(x_1)}}(y_1)-J_{\xi N_{\Phi(x_2)}}(y_1)\|+\|J_{\xi N_{\Phi(x_2)}}(y_1)-J_{\xi N_{\Phi(x_2)}}(y_2)\|\\
		&\leq& l\|x_1-x_2\|+\|y_1-y_2\|\\
		&\leq& \beta \|x_1-x_2\|+\|y_1-y_2\|\\
		& \leq & \|w_1-w_2\|_1.
	\end{eqnarray*}
	On the other hand, using \eqref{contractionproj}, the nonexpansiveness of the reflected resolvent, and \eqref{ResNoCoProj}
	one obtains
	\begin{eqnarray*}
		&&\|(2J_{\xi N_{\Phi(x_1)}}(y_1)-y_1)-(2J_{\xi N_{\Phi(x_2)}}(y_2)-y_2)\|\\
		&\leq&\|(2J_{\xi N_{\Phi(x_1)}}(y_1)-y_1)-(2J_{\xi N_{\Phi(x_1)}}(y_2)-y_2)\|+2\|J_{\xi N_{\Phi(x_1)}}(y_2)-J_{\xi N_{\Phi(x_2)}}(y_2)\|\\
		&=& \|R_{\xi N_{\Phi(x_1)}}(y_1)-R_{\xi N_{\Phi(x_1)}}(y_2)\|+2\|J_{\xi N_{\Phi(x_1)}}(y_2)-J_{\xi N_{\Phi(x_2)}}(y_2)\|\\
		&\leq&\|y_1-y_2\|+2l\|x_1-x_2\|\\
		&\leq& \beta \|x_1-x_2\|+\|y_1-y_2\|\\
		&\leq& \|w_1-w_2\|_1.
	\end{eqnarray*}
Therefore,
\begin{eqnarray}\label{DifRho}
\|\rho(w_1)-\rho(w_2)\|_1 & = &\alpha \|J_{\xi N_{\Phi(x_1)}}(y_1)-J_{\xi N_{\Phi(x_2)}}(y_2)\|+ \nonumber\\
		& + & \|R_{\xi T}(2J_{\xi N_{\Phi(x_1)}}(y_1)-y_1)-R_{\xi T}(2J_{\xi N_{\Phi(x_2)}}(y_2)-y_2)\| + \nonumber\\
		&+& \beta \|P_C(z_1)-P_C(z_2)\| \nonumber\\
		&\leq &  (\alpha + L_{\xi T}) \|w_1-w_2\|_1 + \beta \|z_1-z_2\| \\
		&\leq & (\alpha + L_{\xi T}) \|w_1-w_2\|_1 + \alpha L_{\xi T} \|z_1-z_2\| \nonumber\\
		& \leq & (\alpha + 2L_{\xi T}) \|w_1-w_2\|_1  = \delta\|w_1-w_2\|_1\nonumber
\end{eqnarray}
Then, one obtains the contraction of $\rho$, and utilizing the Banach fixed-point theorem, there is a unique triplex that satisfies $\rho(z^*,y^*,x^*)=(z^*,y^*,x^*)$. We reached the conclusion using Proposition \ref{problemsolutionlemma1}.
\end{proof}

\begin{remark}
	A possible but not unique choice for $\alpha$ and $\beta$ should be
	
	$$\beta = 2l, \,\, \alpha =  L_{\xi T}$$
	
	where $L_{\xi T} < \frac  13 $ (this last condition is satisfied if $\gamma < \frac 5 4, $ i.e. $L< \frac 5 4 \mu$) and $l < \frac 12  L_{\xi T}^2$.  
\end{remark}

\subsection{Algorithmic details}
This subsection introduces a Douglas-Rachford splitting algorithm, referred to as Algorithm 1. We also present the specifics of the definition-based algorithm mentioned in \cite{BianchiMiglierinaMaede} regarding the original problem's form, labeled as \eqref{QVIwithNonCons3}, under the heading Algorithm 2. At the end, we will have the convergence result of the sequence obtained by Algorithm 1.

\begin{figure}[H]
\begin{minipage}{\textwidth}
\begin{algorithm}[H]
\caption{Douglas–Rachford Splitting Algorithm for Projected Solution of Quasi-Variational Inequality with Non-Self Constraints}\label{SecondAlg-Douglas–Rachford}
\begin{algorithmic}
		\STATE \textbf{Initialize:}
		Choose the starting points $x_0\in C$, $y_0\in \mathcal{H}$, and the parameter $\xi = \frac{1}{L} >0$.
		\FOR{k= 0, 1, 2, ...}
		\STATE\textbf{Step 1.}
		\STATE $z_{k+1}= P_{\Phi(x_{k})}(y_{k})$
		\STATE $y_{k+1} = 2(I+\xi T)^{-1}(2z_{k+1}-y_k) - (2z_{k+1}-y_k)$ 
		\STATE $x_{k+1} = P_C(z_{k+1})$
		\STATE\textbf{Step 2.}
		\IF{$y_{k+1} == y_k$ \textbf{and} $x_{k+1} == x_k$}
		\STATE \textbf{break}
		\ELSE
		\STATE $x_k= x_{k+1}$
		\STATE $y_k= y_{k+1}$
		\ENDIF
		\ENDFOR
		\STATE Finalize and return result
\end{algorithmic}
\end{algorithm}
\end{minipage}
\vspace{1em}
\begin{minipage}{\textwidth}
\begin{algorithm}[H]
\caption{Definition–Based Algorithm for Projected Solution of Quasi-Variational Inequality with Non-Self Constraints}
	\label{SecondAlg-Douglas–Rachford}
\begin{algorithmic}
		\STATE \textbf{Initialize:}
		Choose the starting points $x_0\in C$, $y_0\in \Phi(C)$, and the parameter $\gamma>0$.
		\FOR{k= 0, 1, 2, ...}
		\STATE\textbf{Step 1.}
		\STATE $y_{k+1}= P_{\Phi(x_{k})}(y_{k} - \gamma T(y_k))$
		\STATE\textbf{Step 2.}
		\IF{$y_{k+1} == y_k$}
		\STATE $x_{k+1} = P_C(y_{k+1})$
		\IF{$x_{k+1} == x_k$}
		\STATE \textbf{break}
		\ELSE
		\STATE $x_k= x_{k+1}$
		\STATE $y_k= y_{k+1}$
		\STATE{Go to Step 3}
		\ENDIF
		\ELSE
		\STATE $y_k= y_{k+1}$
		\STATE{Go to Step 1}
		\ENDIF
		\STATE\textbf{Step 3.}
		\IF{$y_k$ not in $\Phi(x_k)$}
		\STATE{Select a random point from $\Phi(x_k)$ labeled $y_k$ and go to Step 1}
		\ENDIF
		\ENDFOR
		\STATE Finalize and return result
\end{algorithmic}
\end{algorithm}
\end{minipage}
\end{figure}
\quad\\
Then we have the following convergence result.
\begin{theorem}
	Suppose that (A1)-(A4) are valid. Then by considering 
	\begin{eqnarray*}
		\alpha <(1-2\sqrt{\frac{\gamma -1}{\gamma+1}}),\quad\text{where}\quad \gamma:=\frac{L}{\mu},
	\end{eqnarray*}
	the problem \eqref{QVIwithNonCons1} has a unique projected solution. Additionally, the sequences $(z_k)$, $(y_k)$, and $(x_k)$ generated by Algorithm 1 converge strongly to $z^*$, $y^*$, and $x^*$, respectively where $(z^*, y^*, x^*)$ satisfies \eqref{equationsofsolution}. We also have the following estimation
	\begin{eqnarray*}
\alpha\|z_{k+1}-z^*\|+\|y_k-y^*\|+\beta\|x_k-x^*\|\leq \delta^k(\alpha\|z_1-z^*\|+\|y_0-y^*\|+\beta\|x_0-x^*\|),
	\end{eqnarray*}
	where
$
		\delta:=\alpha+2\sqrt{\frac{\gamma -1}{\gamma+1}}< 1.
$
\end{theorem}

\begin{proof}
	Let consider $w_k=\alpha\|z_{k+1}-z^*\|+\|y_k-y^*\|+\beta\|x_k-x^*\|,~ k\geq 1$. Based on \eqref{DifRho}, we have $w_k\leq\delta w_{k-1}$, and this leads to the conclusion.
\end{proof}

\subsection{Numerical experiment}
This part highlights the efficiency and speed of Douglas-Rachford splitting algorithm compared to the Algorithm 2 through a numerical example. 

\begin{example}
	\label{ex:Aussel}Let $C=\{x=(x^1,x^2) \in \R^2: 0\le x^1 \le 1, 0\le x^2\le 1, x^1+x^2 \ge 1 \}$,
	$$
	\Phi(x)= Q +\frac 1 {2^6} x,
	$$
	where $Q=\{(x^1,x^2) \in \R^2: 0\le x^1 \le 1, 0\le x^2\le 1 \}$ is the unit square. And let
	\begin{eqnarray*}
		Tx = \begin{pmatrix} 0.22 & 0  \\
			0 & 0.25 \\
		\end{pmatrix}x.
	\end{eqnarray*}
It is easy to check that $T$ is $0.25$-Lipschitz continuous and $0.22$-strongly monotone and the set $C$ and the set-valued map $\Phi$ satisfy the assumptions (A1) and (A2), respectively. In the Appendix, the fulfillment of condition (A4) is rigorously assessed.

	 Considering 
	$$
	l=\frac{1}{2^6},\quad L= 0.25,\quad \mu=0.22,\quad and\quad \xi=4,
	$$
	the inequalities \eqref{ParametersAssums} hold. The unique point that meet \eqref{QVIwithNonCons3} is $(\frac{1}{2}, \frac{1}{2})$, since there exists $(z_0^1, z_0^2)=(\frac{1}{2^7}, \frac{1}{2^7})$ such that $(z_0^1, z_0^2)= P_{\Phi(\frac{1}{2}, \frac{1}{2})} (0.12  z_0^1, 0)$, and then $(\frac{1}{2}, \frac{1}{2})=P_C(\frac{1}{2^7}, \frac{1}{2^7})$. Hence we can say the point $(\frac{1}{2}, \frac{1}{2})$ is the unique projected solution.
	
	The table \ref{table:algorithm_comparison} compares Algorithms 1 and 2 in terms of the number of iterations and CPU time required to achieve an accuracy, starting from the same initial points. We set the accuracy at \(10^{-8}\), signifying that approximate projected solutions \(x^*\) and approximate Stampacchia variational inequality solutions  \(z^*\) are accepted as close to exact solutions within \(10^{-8}\). All experiments applied on
	Intel(R) Core(TM) i5-1135G7 CPU, 8GB RAM and Windows 10 Enterprise operating system. 

	\begin{table}[h!]
		\centering
		\begin{tabular}{l|l|c|c}
			\hline
			\textbf{Initial Points} & \textbf{Metric}          & \textbf{Algorithm 1 } & \textbf{Algorithm 2 } \\ \hline
			\multirow{3}{*}{$x_0=(0,1),$ $y_0=(0,1)$} & CPU Time (seconds)    & 0.0                        & 0.015625                       \\ \cline{2-4} 
			& Iterations to Converge & 8                          & 20                          \\ \hline
			\multirow{3}{*}{$x_0=(1,0),$ $y_0=(1,1)$} & CPU Time (seconds)    & 0.0                       & 0.015625                      \\ \cline{2-4} 
			& Iterations to Converge & 8                          & 22                         \\ \hline
			\multirow{3}{*}{$x_0=(0.5,0.75),$ $y_0=(0.5,1)$} & CPU Time (seconds)    & 0.0                        & 0.015625                        \\ \cline{2-4} 
			& Iterations to Converge & 8                          & 15                          \\ \hline
		\end{tabular}
		\caption{Comparison of Algorithm 1 and Algorithm 2 with the Same Initial Points but Different Initializations}
		\label{table:algorithm_comparison}
	\end{table}
	
\end{example}

\section{Conclusions}	

The contributions of this work are significant both theoretically and practically. Theoretically, we have extended the applicability of the Douglas-Rachford splitting algorithm to a broader class of problems by addressing quasi-variational inequalities with non-self constraint maps, which has not been extensively studied in existing literature. 

Practically, our numerical experiment emphasizes the efficiency and reliability of the proposed method. By comparing it with the definition-based approach, we have shown that our algorithm can achieve superior performance. This highlights the algorithm's utility in real-world scenarios, where the constraints and conditions can often be complex and non-standard.

Future work may explore further refinements and extensions of the algorithm, such as adapting it to other types of variational inequalities and constraint maps, or investigating its performance in different Hilbert space settings. Additionally, examining the impact of various parameter choices on the algorithm's performance could yield valuable guidelines for practical implementations.

\section{Appendix}
Following, the computation verifying the satisfaction of condition (A4) in Example 2 is provided.

For every $x=(x^1, x^2), y=(y^1,y^2) \in C$ there are three situations:

\begin{flalign*}
	\begin{cases}
		(1)\quad (x^1, x^2) = (y^1, y^2). \\
		(2)\quad  x^1 < y^1, \text{ and}~ y^2< x^2.\\
		(3)\quad  x^1 < y^1, \text{ and}~ x^2< y^2.
	\end{cases}&&
\end{flalign*}
In situation (1), for every $z=(z^1,z^2)\in\R^2$, $P_{\Phi(x)}(z)= P_{\Phi(y)}(z)$ that satisfies (A4). We will now discuss how condition (A4) is satisfied in situation (2). 

\begin{flalign*}
	P_{\Phi(x)}(z) = 
	\begin{cases}
		1.~ z , & \text{if } \frac{y^1}{2^6}\leq z^1\leq 1 + \frac{x^1}{2^6}, \frac{x^2}{2^6}\leq z^2\leq 1+ \frac{y^2}{2^6} \\
		\begin{aligned}
			&2.~(1 + \frac{x^1}{2^6}, 1 + \frac{x^2}{2^6}),
		\end{aligned}
		& \begin{aligned}
			&\text{if } 1+ \frac{y^1}{2^6} \leq z^1,  z^2\geq 1 + \frac{x^2}{2^6}, 
		\end{aligned}\\
		\begin{aligned}
			&3.~(1 + \frac{x^1}{2^6}, 1 + \frac{x^2}{2^6}),
		\end{aligned}
		& \begin{aligned}
			&\text{if } 1+ \frac{x^1}{2^6} \leq z^1\leq  1+ \frac{y^1}{2^6},  z^2\geq 1 + \frac{x^2}{2^6}, 
		\end{aligned}\\
		\begin{aligned}
			&4.~(1 + \frac{x^1}{2^6}, z^2),
		\end{aligned}
		& \begin{aligned}
			&\text{if }   z^1\geq 1 + \frac{y^1}{2^6}, 1+ \frac{y^2}{2^6} \leq z^2\leq  1+ \frac{x^2}{2^6},
		\end{aligned}\\
		\begin{aligned}
			&5.~(1 + \frac{x^1}{2^6}, z^2),
		\end{aligned}
		& \begin{aligned}
			&\text{if }   1+ \frac{x^1}{2^6} \leq z^1\leq  1+ \frac{y^1}{2^6}, 1+ \frac{y^2}{2^6} \leq z^2\leq  1+ \frac{x^2}{2^6},
		\end{aligned}\\
		\begin{aligned}
			&6.~(z^1, 1+ \frac{x^2}{2^6}),
		\end{aligned}
		& \begin{aligned}
			&\text{if }   \frac{y^1}{2^6} \leq z^1\leq  1+ \frac{x^1}{2^6}, z^2\geq  1+ \frac{x^2}{2^6},
		\end{aligned}\\
		\begin{aligned}
			&7.~ z,
		\end{aligned}
		& \begin{aligned}
			&\text{if }   \frac{y^1}{2^6} \leq z^1\leq  1+ \frac{x^1}{2^6}, 1+ \frac{y^2}{2^6}\leq z^2\leq  1+ \frac{x^2}{2^6},
		\end{aligned}\\
		\begin{aligned}
			&8.~z,
		\end{aligned}
		& \begin{aligned}
			&\text{if }   \frac{x^1}{2^6} \leq z^1\leq \frac{y^1}{2^6}, 1+ \frac{y^2}{2^6}\leq z^2\leq  1+ \frac{x^2}{2^6},
		\end{aligned}\\
		\begin{aligned}
			&9.~(z^1,1+ \frac{x^2}{2^6}),
		\end{aligned}
		& \begin{aligned}
			&\text{if }   \frac{x^1}{2^6} \leq z^1\leq \frac{y^1}{2^6},  z^2\geq  1+ \frac{x^2}{2^6},
		\end{aligned}\\
		\begin{aligned}
			&10.~(\frac{x^1}{2^6}, 1+ \frac{x^2}{2^6}),
		\end{aligned}
		& \begin{aligned}
			&\text{if }   \frac{x^1}{2^6} \geq z^1,  z^2\geq  1+ \frac{x^2}{2^6},
		\end{aligned}\\
		\begin{aligned}
			&11.~(\frac{x^1}{2^6}, z^2),
		\end{aligned}
		& \begin{aligned}
			\text{if }   \frac{x^1}{2^6} \geq z^1,  1+ \frac{y^2}{2^6}\leq z^2\leq  1+ \frac{x^2}{2^6},
		\end{aligned}\\
		\begin{aligned}
			&12.~(\frac{x^1}{2^6}, z^2),
		\end{aligned}
		& \begin{aligned}
			\text{if }   \frac{x^1}{2^6} \geq z^1,  \frac{x^2}{2^6}\leq z^2\leq  1+ \frac{y^2}{2^6},
		\end{aligned}\\
		\begin{aligned}
			&13.~z,
		\end{aligned}
		& \begin{aligned}
			\text{if }   \frac{x^1}{2^6} \leq z^1\leq \frac{y^1}{2^6},  \frac{x^2}{2^6}\leq z^2\leq  1+ \frac{y^2}{2^6},
		\end{aligned}\\
		\begin{aligned}
			&14.~(\frac{x^1}{2^6}, \frac{x^2}{2^6}),
		\end{aligned}
		& \begin{aligned}
			\text{if }   \frac{x^1}{2^6} \geq z^1,  \frac{y^2}{2^6}\leq z^2\leq  \frac{x^2}{2^6},
		\end{aligned}\\
		\begin{aligned}
			&15.~(z^1, \frac{x^2}{2^6}),
		\end{aligned}
		& \begin{aligned}
			\text{if }   \frac{x^1}{2^6} \leq z^1\leq\frac{y^1}{2^6},  \frac{y^2}{2^6}\leq z^2\leq\frac{x^2}{2^6},
		\end{aligned}\\
		\begin{aligned}
			&16.~(z^1,\frac{x^2}{2^6}),
		\end{aligned}
		& \begin{aligned}
			\text{if }   \frac{y^1}{2^6} \leq z^1\leq 1+\frac{x^1}{2^6}, \frac{y^2}{2^6}\leq z^2\leq \frac{x^2}{2^6},
		\end{aligned}\\
		\begin{aligned}
			&17.~(1+\frac{x^1}{2^6},1+\frac{x^2}{2^6}),
		\end{aligned}
		& \begin{aligned}
			\text{if }   1+\frac{x^1}{2^6} \leq z^1\leq 1+\frac{y^1}{2^6}, \frac{y^2}{2^6}\leq z^2\leq \frac{x^2}{2^6},
		\end{aligned}\\
		\begin{aligned}
			&18.~(\frac{x^1}{2^6}, \frac{x^2}{2^6}),
		\end{aligned}
		& \begin{aligned}
			\text{if }   z^1\leq\frac{x^1}{2^6},   z^2\leq \frac{y^2}{2^6},
		\end{aligned}\\
		\begin{aligned}
			&19.~(z^1,\frac{x^2}{2^6}),
		\end{aligned}
		& \begin{aligned}
			\text{if }   \frac{x^1}{2^6} \leq z^1\leq \frac{y^1}{2^6},  z^2\leq \frac{y^2}{2^6},
		\end{aligned}\\
		\begin{aligned}
			&20.~(z^1,\frac{x^2}{2^6}),
		\end{aligned}
		& \begin{aligned}
			\text{if }   \frac{y^1}{2^6} \leq z^1\leq 1+\frac{x^1}{2^6},  z^2\leq \frac{y^2}{2^6},
		\end{aligned}\\
		\begin{aligned}
			&21.~(1+\frac{x^1}{2^6}, \frac{x^2}{2^6}),
		\end{aligned}
		& \begin{aligned}
			\text{if }   1+\frac{x^1}{2^6} \leq z^1\leq 1+\frac{y^1}{2^6},   z^2\leq \frac{y^2}{2^6},
		\end{aligned}\\
		\begin{aligned}
			&22.~(1+\frac{x^1}{2^6}, \frac{x^2}{2^6}),
		\end{aligned}
		& \begin{aligned}
			\text{if }   1+\frac{y^1}{2^6} \leq z^1,  z^2\leq \frac{y^2}{2^6},
		\end{aligned}\\
		\begin{aligned}
			&23.~(1+\frac{x^1}{2^6}, \frac{x^2}{2^6}),
		\end{aligned}
		& \begin{aligned}
			\text{if }   1+\frac{y^1}{2^6} \leq z^1,  \frac{y^2}{2^6}\leq z^2\leq \frac{x^2}{2^6},
		\end{aligned}\\
		\begin{aligned}
			&24.~(1+\frac{x^1}{2^6}, z^2),
		\end{aligned}
		& \begin{aligned}
			\text{if }   1+\frac{x^1}{2^6} \leq z^1\leq  1+\frac{y^1}{2^6},  \frac{x^2}{2^6}\leq z^2\leq 1+\frac{y^2}{2^6},
		\end{aligned}\\
		\begin{aligned}
			&25.~(1+\frac{x^1}{2^6}, z^2),
		\end{aligned}
		& \begin{aligned}
			\text{if }   1+\frac{y^1}{2^6} \leq z^1,  \frac{x^2}{2^6}\leq z^2\leq 1+\frac{y^2}{2^6}.
		\end{aligned}\\
	\end{cases}&& 
\end{flalign*}

\begin{flalign*}
	P_{\Phi(y)}(z) = 
	\begin{cases}
		1.~ z , & \text{if } \frac{y^1}{2^6}\leq z^1\leq 1 + \frac{x^1}{2^6}, \frac{x^2}{2^6}\leq z^2\leq 1+ \frac{y^2}{2^6}, \\
		\begin{aligned}
			&2.~(1 + \frac{y^1}{2^6}, 1 + \frac{y^2}{2^6}),
		\end{aligned}
		& \begin{aligned}
			&\text{if } 1+ \frac{y^1}{2^6} \leq z^1,  z^2\geq 1 + \frac{x^2}{2^6}, 
		\end{aligned}\\
		\begin{aligned}
			&3.~(z^1, 1 + \frac{y^2}{2^6}),
		\end{aligned}
		& \begin{aligned}
			&\text{if } 1+ \frac{x^1}{2^6} \leq z^1\leq  1+ \frac{y^1}{2^6},  z^2\geq 1 + \frac{x^2}{2^6}, 
		\end{aligned}\\
		\begin{aligned}
			&4.~(1 + \frac{y^1}{2^6}, 1 + \frac{y^2}{2^6}),
		\end{aligned}
		& \begin{aligned}
			&\text{if }   z^1\geq 1 + \frac{y^1}{2^6}, 1+ \frac{y^2}{2^6} \leq z^2\leq  1+ \frac{x^2}{2^6},
		\end{aligned}\\
		\begin{aligned}
			&5.~(z^1, 1 + \frac{y^2}{2^6}),
		\end{aligned}
		& \begin{aligned}
			&\text{if }   1+ \frac{x^1}{2^6} \leq z^1\leq  1+ \frac{y^1}{2^6}, 1+ \frac{y^2}{2^6} \leq z^2\leq  1+ \frac{x^2}{2^6},
		\end{aligned}\\
		\begin{aligned}
			&6.~(z^1, 1+ \frac{y^2}{2^6}),
		\end{aligned}
		& \begin{aligned}
			&\text{if }   \frac{y^1}{2^6} \leq z^1\leq  1+ \frac{x^1}{2^6}, z^2\geq  1+ \frac{x^2}{2^6},
		\end{aligned}\\
		\begin{aligned}
			&7.~(z^1, 1+ \frac{y^2}{2^6}),
		\end{aligned}
		& \begin{aligned}
			&\text{if }   \frac{y^1}{2^6} \leq z^1\leq  1+ \frac{x^1}{2^6}, 1+ \frac{y^2}{2^6}\leq z^2\leq  1+ \frac{x^2}{2^6},
		\end{aligned}\\
		\begin{aligned}
			&8.~(\frac{y^1}{2^6}, 1+ \frac{y^2}{2^6}),
		\end{aligned}
		& \begin{aligned}
			&\text{if }   \frac{x^1}{2^6} \leq z^1\leq \frac{y^1}{2^6}, 1+ \frac{y^2}{2^6}\leq z^2\leq  1+ \frac{x^2}{2^6},
		\end{aligned}\\
		\begin{aligned}
			&9.~(\frac{y^1}{2^6}, 1+ \frac{y^2}{2^6}),
		\end{aligned}
		& \begin{aligned}
			&\text{if }   \frac{x^1}{2^6} \leq z^1\leq \frac{y^1}{2^6},  z^2\geq  1+ \frac{x^2}{2^6},
		\end{aligned}\\
		\begin{aligned}
			&10.~(\frac{y^1}{2^6}, 1+ \frac{y^2}{2^6}),
		\end{aligned}
		& \begin{aligned}
			&\text{if }   \frac{x^1}{2^6} \geq z^1,  z^2\geq  1+ \frac{x^2}{2^6},
		\end{aligned}\\
		\begin{aligned}
			&11.~(\frac{y^1}{2^6}, 1+ \frac{y^2}{2^6}),
		\end{aligned}
		& \begin{aligned}
			\text{if }   \frac{x^1}{2^6} \geq z^1,  1+ \frac{y^2}{2^6}\leq z^2\leq  1+ \frac{x^2}{2^6},
		\end{aligned}\\
		\begin{aligned}
			&12.~(\frac{y^1}{2^6}, z^2),
		\end{aligned}
		& \begin{aligned}
			\text{if }   \frac{x^1}{2^6} \geq z^1,  \frac{x^2}{2^6}\leq z^2\leq  1+ \frac{y^2}{2^6},
		\end{aligned}\\
		\begin{aligned}
			&13.~(\frac{y^1}{2^6}, z^2),
		\end{aligned}
		& \begin{aligned}
			\text{if }   \frac{x^1}{2^6} \leq z^1\leq \frac{y^1}{2^6},  \frac{x^2}{2^6}\leq z^2\leq  1+ \frac{y^2}{2^6},
		\end{aligned}\\
		\begin{aligned}
			&14.~(\frac{y^1}{2^6}, z^2),
		\end{aligned}
		& \begin{aligned}
			\text{if }   \frac{x^1}{2^6} \geq z^1,  \frac{y^2}{2^6}\leq z^2\leq  \frac{x^2}{2^6},
		\end{aligned}\\
		\begin{aligned}
			&15.~(\frac{y^1}{2^6}, z^2),
		\end{aligned}
		& \begin{aligned}
			\text{if }   \frac{x^1}{2^6} \leq z^1\leq\frac{y^1}{2^6},  \frac{y^2}{2^6}\leq z^2\leq\frac{x^2}{2^6},
		\end{aligned}\\
		\begin{aligned}
			&16.~z,
		\end{aligned}
		& \begin{aligned}
			\text{if }   \frac{y^1}{2^6} \leq z^1\leq 1+\frac{x^1}{2^6},   \frac{y^2}{2^6}\leq z^2\leq \frac{x^2}{2^6},
		\end{aligned}\\
		\begin{aligned}
			&17.~z,
		\end{aligned}
		& \begin{aligned}
			\text{if }   1+\frac{x^1}{2^6} \leq z^1\leq 1+\frac{y^1}{2^6}, \frac{y^2}{2^6}\leq z^2\leq \frac{x^2}{2^6},
		\end{aligned}\\
		\begin{aligned}
			&18.~(\frac{y^1}{2^6},\frac{y^2}{2^6}),
		\end{aligned}
		& \begin{aligned}
			\text{if }    z^1\leq \frac{x^1}{2^6},  z^2\leq \frac{y^2}{2^6},
		\end{aligned}\\
		\begin{aligned}
			&19.~(\frac{y^1}{2^6},\frac{y^2}{2^6}),
		\end{aligned}
		& \begin{aligned}
			\text{if }   \frac{x^1}{2^6} \leq z^1\leq \frac{y^1}{2^6},  z^2\leq \frac{y^2}{2^6},
		\end{aligned}\\
		\begin{aligned}
			&20.~(z^1,\frac{y^1}{2^6}),
		\end{aligned}
		& \begin{aligned}
			\text{if }   \frac{y^1}{2^6} \leq z^1\leq 1+\frac{x^1}{2^6}, z^2\leq \frac{y^2}{2^6},
		\end{aligned}\\
		\begin{aligned}
			&21.~(z^1, \frac{y^2}{2^6}),
		\end{aligned}
		& \begin{aligned}
			\text{if }   1+\frac{x^1}{2^6} \leq z^1 \leq 1+\frac{y^1}{2^6},   z^2\leq \frac{y^2}{2^6},
		\end{aligned}\\
		\begin{aligned}
			&22.~(1+\frac{y^1}{2^6}, \frac{y^2}{2^6}),
		\end{aligned}
		& \begin{aligned}
			\text{if }   1+\frac{y^1}{2^6} \leq z^1,  z^2\leq \frac{y^2}{2^6},
		\end{aligned}\\
		\begin{aligned}
			&23.~(1+\frac{y^1}{2^6}, z^2),
		\end{aligned}
		& \begin{aligned}
			\text{if }   1+\frac{y^1}{2^6} \leq z^1,  \frac{y^2}{2^6}\leq z^2\leq 1+\frac{x^2}{2^6},
		\end{aligned}\\
		\begin{aligned}
			&24.~z,
		\end{aligned}
		& \begin{aligned}
			\text{if }   1+\frac{x^1}{2^6} \leq z^1 \leq 1+\frac{y^1}{2^6},  \frac{x^2}{2^6}\leq z^2\leq 1+\frac{y^2}{2^6},
		\end{aligned}\\
		\begin{aligned}
			&25.~(1+\frac{y^1}{2^6}, z^2),
		\end{aligned}
		& \begin{aligned}
			\text{if }   1+\frac{y^1}{2^6} \leq z^1,  \frac{x^2}{2^6}\leq z^2\leq 1+\frac{y^2}{2^6}.
		\end{aligned}\\
	\end{cases} &&		
\end{flalign*}\\
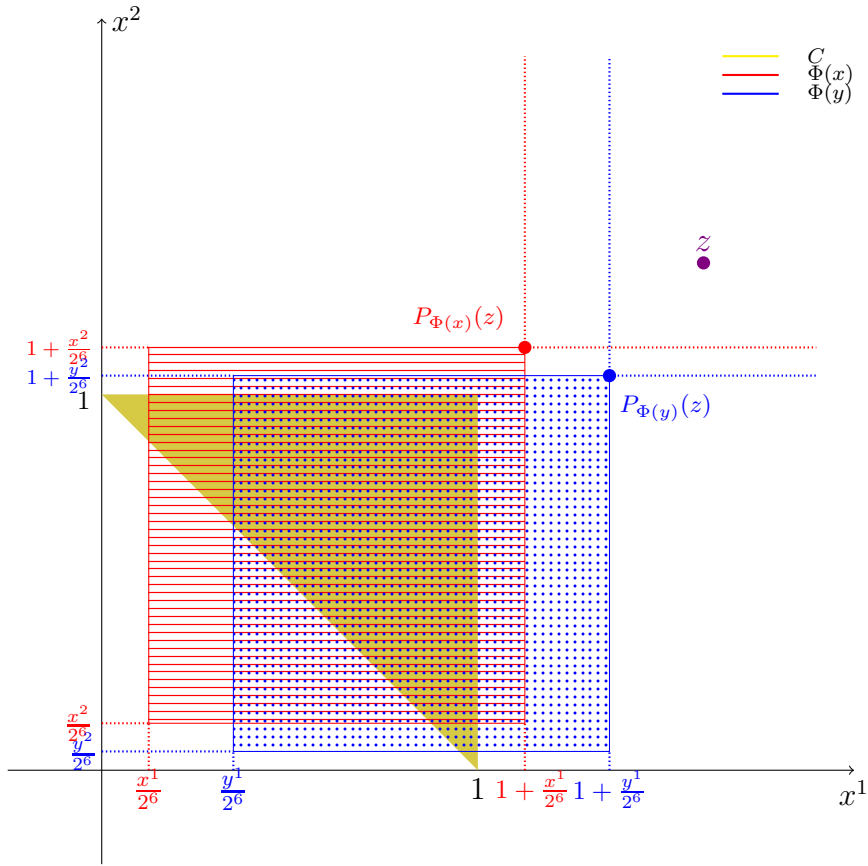
\begin{figure}
	\begin{tikzpicture}[scale=2.5]
		
		\draw[->](-0.5,0)--(4,0)node[below]{\fontsize{12}{4}{\( x^1 \)}};
		\draw[->](0,-0.5)--(0,4)node[right]{\fontsize{12}{4}{\( x^2 \)}};
		
		
		\fill [fill=yellow!80!black]
		(2,0) 
		-- (2,2)  
		-- (0,2);
		\node at (2, -0.1){\fontsize{12}{5}\selectfont{1}};
		\node at (-0.095, 1.97){\fontsize{10}{5}\selectfont{1}};
		
		
		\draw[red] (0.25,0.25) rectangle (2.25,2.25);
		\fill[pattern=horizontal lines, pattern color=red] (0.25,0.25) rectangle (2.25,2.25);
		
		\draw[dashed, dash pattern=on 0.55pt off 1pt, red, thick] (0.25,0)--(0.25, 0.25);
		\node at (0.25, -0.1) {\fontsize{12}{4}\selectfont \textcolor{red}{$\frac{x^1}{2^6}$}};
		
		\draw[dashed, dash pattern=on 0.55pt off 1pt, red, thick] (2.25,0)--(2.25, 0.25);
		\node at (2.29, -0.1) {\fontsize{10}{4}\selectfont \textcolor{red}{$1+\frac{x^1}{2^6}$}};
		
		\draw[dashed, dash pattern=on 0.55pt off 1pt, red, thick] (0, 0.25)--(0.25, 0.25);
		\node at (-0.15, 0.25) {\fontsize{10}{4} \textcolor{red}{$\frac{x^2}{2^6}$}};
		
		\draw[dashed, dash pattern=on 0.55pt off 1pt, red, thick] (0, 2.25)--(0.25, 2.25);
		\node at (-0.25, 2.25) {\fontsize{8}{4} \textcolor{red}{$1+\frac{x^2}{2^6}$}};
		
		
		\draw[blue] (0.7,0.1) rectangle (2.7,2.1);
		\fill[pattern=dots, pattern color=blue] (0.7,0.1) rectangle (2.7,2.1);
		
		\draw[dashed, dash pattern=on 0.55pt off 1pt, blue, thick] (0.7,0)--(0.7, 0.1);
		\node at (0.7, -0.1) {\fontsize{12}{4}\selectfont \textcolor{blue}{$\frac{y^1}{2^6}$}};
		
		\draw[dashed, dash pattern=on 0.55pt off 1pt, blue, thick] (2.7,0)--(2.7, 0.1);
		\node at (2.7, -0.1) {\fontsize{10}{4}\selectfont \textcolor{blue}{$1+\frac{y^1}{2^6}$}};
		
		\draw[dashed, dash pattern=on 0.55pt off 1pt, blue, thick] (0, 0.1)--(0.7, 0.1);
		\node at (-0.1, 0.1) {\fontsize{10}{4}\selectfont \textcolor{blue}{$\frac{y^2}{2^6}$}};
		
		\draw[dashed, dash pattern=on 0.55pt off 1pt, blue, thick] (0, 2.1)--(0.7, 2.1);
		\node at (-0.25, 2.1) {\fontsize{8}{4} \textcolor{blue}{$1+\frac{y^2}{2^6}$}};
		
		
		\draw[dashed, dash pattern=on 0.55pt off 1pt, red, thick] (2.25, 2.25)--(2.25, 3.8);
		\draw[dashed, dash pattern=on 0.55pt off 1pt, red, thick] (2.25, 2.25)--(3.8, 2.25);
		
		\draw[dashed, dash pattern=on 0.55pt off 1pt, blue, thick] (2.7,2.1)--(2.7, 3.8);
		\draw[dashed, dash pattern=on 0.55pt off 1pt, blue, thick] (2.7,2.1)--(3.8, 2.1);
		
		
		\fill[violet] (3.2,2.7) circle (1pt);
		\node at (3.2, 2.8) {\fontsize{13}{4}\selectfont \textcolor{violet}{$z$}};
		
		
		\fill[red] (2.25,2.25) circle (1pt);
		\node at (1.9, 2.4) {\fontsize{9}{4}\selectfont \textcolor{red}{$P_{\Phi(x)}(z)$}};
		
		\fill[blue] (2.7,2.1) circle (1pt);
		\node at (3, 1.93) {\fontsize{9}{4}\selectfont \textcolor{blue}{$P_{\Phi(y)}(z)$}};
		
		\draw[yellow, thick, pattern=horizontal lines] (3.3,3.8) -- (3.6,3.8);
		\node[right] at (3.7,3.8) {\fontsize{8}{5}\selectfont \textcolor{black}{$C$}};
		
		\draw[red, thick, pattern=horizontal lines] (3.3,3.7) -- (3.6,3.7);
		\node[right] at (3.7,3.7) {\fontsize{8}{5}\selectfont \textcolor{black}{$\Phi(x)$}};
		
		\draw[blue, thick, pattern=dots] (3.3,3.6) -- (3.6,3.6);
		\node[right] at (3.7,3.6) {\fontsize{8}{5}\selectfont \textcolor{black}{$\Phi(y)$}};

	\end{tikzpicture}
	\caption{The figure illustrates case 2 within the framework of situation 2.}
\end{figure}

After a thorough calculation has been completed for situation 2, comparable calculations can be carried out for situation 3. Based on the analysis of all 25 cases in situations (2) and (3), we can easily conclude that:$$\|P_{\Phi(x)}(z) - P_{\Phi(y)}(z)\|\leq \frac{1}{2^6}\sqrt{(x^1-y^1)^2+(x^2-y^2)^2}. $$And it implies that condition (A4) is satisfied.\end{document}